\documentclass[english]{amsart}
\usepackage{preamble}

\begin{document}

\title{On the Non-Existence of J-Homomorphisms of Higher Height}

\date{\today}
\maketitle

\begin{abstract}
We show that for a large class of connective spectra, the mapping space into the units of the $p$-complete sphere spectrum is insensitive to $L_1$-localization followed by taking connective cover.
\end{abstract}

\tableofcontents{}

\section{Introduction}
\subsection{Background and Motivation}
Let $\Sph$ be the sphere spectrum and $\Ort$ be the space of infinite orthogonal matrices, that is, the colimit of the Lie groups $\Ort_n$ for $n\to \infty$.
The classical $J$-homomorphism, introduced by Whitehead \cite{whiteheadJhomomorphism}, is a map $\pi_*\Ort \to \pi_*\Sph$ constructed from the linear actions of orthogonal matrices on spheres. The $J$-homomorphism is almost a map of graded groups, except for $\pi_0$, where it is the map $\ZZ/2\to \ZZ$ sending $1$ to $-1$. This deficient behavior on $\pi_0$ may be explained by the fact that the $J$-homomorphism comes from a map of spectra $\Ort \to \Sph^\times$, where $\Sph^\times$ is the connective spectrum of \emph{units} in the sphere spectrum, and where $\Ort$ is regarded as a connective spectrum via the operation of orthogonal sum of matrices. While the higher homotopy groups of $\Sph^\times$ are the same as that of $\Sph$, the lowest one is $\ZZ/2$, generated by $-1$ under multiplication, which is now compatible with the multiplication of $\pi_0\Ort\simeq \ZZ/2$ via the $J$-homomorphism. 

Starting with the work of Adams,
the image of the $J$-homomorphism has been extensively studied \cite{AdamsJI,AdamsJII,AdamsJIII,mahowald1970order,bruner2022adams}. It constitutes the piece of the stable homotopy groups of spheres detected by the homotopy of the $L_1$-local sphere (see, e.g., \cite[Theorem 7]{JacobChrom35}). As in the discussion \cite[\S 6.5]{pAdicJ}, one could hope to similarly hit elements of higher chromatic height via higher height analogs of the $J$-homomorphism. Focusing on the $p$-complete version of the sphere spectrum, one could try to map into $\Sph_p^\times$ the connective cover of Lubin-Tate theory $E_n$, or the $K$-theory spectrum of a ring spectrum of chromatic height $n \ge 1$, which by the chromatic redshift principle of Ausoni and Rognes detects height $n+1$ information (see \cite{Rognes_Redshift,land2020purity,clausen2020descent,Nulls} for more precise formulations). 

In this paper we show that there is no higher height analog of the $J$-homomorphism. More precisely, under suitable assumptions on a spectrum $X$, a map
$J'\colon X\to \Sph_p^\times$ factors canonically through the connective cover of the $L_1$-localization $\tau_{\ge 0}L_1X$ (and in fact, up to $\pi_0$, through the connective cover of its $K(1)$-localization).  Hence, informally all the information obtained from higher $J$-homomorphisms on the stable homotopy groups of spheres is already seen at height $1$.

\subsection{Main Result}
For $n\ge 0$, let $T(n)$ be the telescope of a $v_n$-self map of a finite spectrum of type $n$, and set $T(\infty):=\FF_p$. Let $\Spch\subseteq \Sp$ be the full subcategory generated under colimits and desuspensions from the $T(n)$-s for $n=0,1\dots,\infty$.
\begin{theorem}\label{itro:main}
Let $X\in \Spch$ be a connective spectrum. Then,    
\[
\Map(X,\Sph_p^\times) \simeq \Hom_\Ab(\pi_0X,\FF_p^\times)\times \Map(\tau_{\ge 0}L_{K(1)}X,(\Sph_p^\times)^\wedge_p).
\]
\end{theorem}

\begin{rem} \label{rem:pic_instead_of_gl}
The result in \Cref{itro:main} could be improved. Conditionally on a work in progress with Burklund and Ramzi, one can lift the ($p$-complete version of) the result from the units $\Sph_p^\times$ to the Picard spectrum $\pic(\Sph_p)$, which is arguably a more natural recipient for the $J$-homomorphism. We shall elaborate on this fact and its generalizations in the mentioned work in progress.  
\end{rem}

We conclude the discussion of our main result by pointing out that
the class of spectra $\Spch$ include several families of spectra of interest. First, it contains all the connective covers of $L_n^f$-local spectra, such as Lubin-Tate theories and the the $K(n)$-local spheres.  By the validity of the Quillen-Lichtenbaum conjectures for the integers (\cite{KahnQLTwo,RogQLTwo,voevodsky2011motivic,WaldKchrom}), $\Spch$ contains also all the $K$-theory spectra of $\ZZ$-algebras (see \Cref{K-theory-disc-chrom}). More generally, a higher height generalization of Quillen-Lichtenbaum conjecture, established by Hahn and Wilson for truncated Brown-Peterson spectra in \cite{hahn2020redshift}, shows that $\Spch$ contains the $K$-theory of arbitrary  $\BP{n}$-algebras, where $\BP{n}$ stands for an $\EE_3$-form of the truncated Brown-Peterson spectra as in \cite{hahn2020redshift} (see \Cref{HW-chromatic}). 
One could optimistically hope that these examples are instances of a general phenomenon, closely related to chromatic redshift.
\begin{conjecture}
Let $R$ be an $\EE_1$-ring spectrum. If $R\in \Spch$ then $K(R)\in \Spch$. 
\end{conjecture}
Thus, if the conjecture holds, then our main theorem applies to the $K$-theory spectra of arbitrary $\EE_1$-rings in $\Spch$.

\subsection{Outline of the Proof} 

First, using the decomposition $\Sph_p^\times \simeq \FF_p^\times \oplus (\Sph_p^\times)^\wedge_p$, the statement is equivalent to $\Map(X,(\Sph_p^\times)^\wedge_p) \simeq \Map(\tau_{\ge 0}L_{K(1)}X,(\Sph_p^\times)^\wedge_p)$. 
The key features of $(\Sph_p^\times)^\wedge_p$ that come into the proof of this isomorphism are the following two:
\begin{enumerate}
\item \label{item:cond_1} There are no maps from $\ZZ$ to its suspension: 
\[
\Map(\ZZ,\Sigma (\Sph_p^\times)^\wedge_p)\simeq \pt.
\]
\item \label{item:cond_2} There are no maps from the reduced suspension spectrum of $B^2C_p$ to it. Equivalently, the pullback map 
\[
\pi^*(\Sph_p^\times)^\wedge_p \to ((\Sph_p^\times)^\wedge_p)^{B^2C_p} 
\]
along the terminal map $\pi\colon B^2C_p \to \pt$
is an isomorphism.
\end{enumerate}
The first property is an immediate consequence of the main result of \cite{carmeli2023strict}. The second follows from a result of Lee, giving the vanishing of the stable cohomotopy of the spaces $B^mC_p$ for $m\ge 2$  (see \cite{lee1992stable}). Interestingly, these two properties both follow, in one way or another, from the Segal conjecture (see \cite{carlsson1984equivariant,lin1980conjectures,adams1992segal}). 

The way these two ingredients combine to give the main result is relatively straight forward: we essentially show that the discrepancy between a connective spectrum in $\Spch$ and the connective cover of its $L_1$-localization can be built from the two spectra $\Sigma^{-1}\ZZ$ and $\Sigma^\infty B^2C_p$ under colimits. In practice, it is convenient to consider such generation questions in the stable settings, allowing colimits \emph{as well as desuspensions}. Thus, we take a somewhat indirect approach. 

First, we show that condition (\ref{item:cond_1}) allows us to extrapolate $(\Sph_p^\times)^\wedge_p$ to negative homotopy degrees and obtain a spectrum $\br$, which by condition (\ref{item:cond_2}) receives no maps from any desuspension of $\Sigma^\infty B^2C_p$. Then, we show that the latter property alone implies that $\br$ satisfies the stable version of the main result: for $X\in \Spch$ we have
\[
\Map(X,\br) \simeq \Map(L_1X,\br).
\] 

This stable version of \Cref{itro:main} follows by analyzing the chromatic behavior of the localizing ideal generated by $\Sigma^\infty B^2C_p$: we show, using the chromatic cyclotomic extensions of \cite{carmeli2021chromatic}, that this localizing ideal contains all the telescopes $T(n)$ for $n\ge 2$, including $\infty$. This allows us to relate co-acyclicity with respect to $\Sigma^\infty B^2C_p$ and co-acyclicity with respect to the various telescopes, giving the stable version above. 
The main result follows from this stable version by restricting to connective $X$, for which maps to $\br$ and $(\Sph_p^\times)^\wedge_p$ coincide.

\subsection{Acknowledgements} The main idea behind this work emerged during the author's participation in the Homotopy Theory workshop at Oberwolfach, and he is grateful to the organizers, administration, and staff for the invitation and hospitality. Parts of the content of this paper, especially the relationship between Eilenberg-MacLane spaces and chromatic localizations, were discussed before with Schlank and Yanovski. The author would like to thank them for these discussions and years of fruitful collaboration, which directly and indirectly influenced this work. He would also like to thank Robert Burklund for several suggestions and comments that simplified some proofs, and Dustin Clausen, Emmanuel Farjoun, Jeremy Hahn, and Ishan Levy for valuable discussions regarding possible applications and interpretations of this work. Finally, he would like to thank Allen Yuan and Robert Burklund for comments on a draft.    
The author is partially supported by the Danish National Research Foundation
through the Copenhagen Centre for Geometry and Topology (DNRF151).

\section{Ideals in Symmetric Monoidal Stable Categories}
We recall the definition and basic properties of localizing ideals in symmetric monoidal stable $\infty$-categories. We warn the reader that we work in the presentable settings, which are different from various contexts for studying localizing ideals in the literature. 

\begin{defn}
Let $\cC\in \calg(\Prst)$. A \tdef{localizing ideal} in $\cC$ is a full presentable subcategory $I\subseteq \cC$ closed under colimits and tensoring with arbitrary objects of $\cC$. We shall indicate that $I$ is a localizing ideal by the notation $I\unlhd \cC$.
The quotient $\cC/I$ is the cofiber of the map $I\into \cC$ in $\Prst$. 
\end{defn}

For a set of objects $S\subseteq \cC$, we denote by $\ideal{S}$ the localizing ideal generated by $S$, i.e., the full subcategory of $\cC$ generated from $S$ under colimits and tensoring with objects of $\cC$. 
\begin{example}\label{ex:Bausfield_Class}
Let $\cC\in \calg(\Prst)$.
For $E\in \cC$, the collection 
\[
\mdef{\ann(E)} := \{X\in \cC : X\otimes E = 0\}
\]
is a localizing ideal, also known as the \tdef{Bausfield class} of $E$. The quotient $\cC/\ann(E)$ is equivalent to the Bausfield localization of $\cC$ with respect to $E$ and denoted $\cC_E$. 
\end{example}

\begin{war}
In various sources the notation $\ideal{E}$ is reserved for the Bausfield class of $E$. Since Bausfield classes will play no role in this work, we hope that this notational conflict will not lead to a confusion. 
\end{war}

For $I\unlhd \cC$, the quotient functor $\cC \to \cC/I$ is a symmetric monoidal localization. The right adjoint of this localization is a fully faithful embedding $\cC/I\into \cC$ whose essential image consists of those $Y\in \cC$ for which $\Map(X,Y)=\pt$ for all $X\in I$.
We denote the composition of the quotient functor $\cC\to \cC/I$ and its right adjoint $\cC/I\into \cC$ by 
\[
\mdef{L_I}\colon \cC \to \cC,
\]
and its image by $\mdef{L_I\cC}\subseteq \cC$. Thus, the composition 
\[
L_I\cC \into \cC \to \cC/I
\]
is an equivalence. 

\begin{war}
Note that the Bausfield localization $L_E$ and our localization $L_{\ideal{E}}$ are different: the kernel of the first is the annihilator $\ann(E)$ and of the other is the  localizing ideal $\ideal{E}$ generated by it.
\end{war}
Ideals in symmetric monoidal stable $\infty$-categories behave much like ordinary ideals of commutative rings. For example, 
if $I,J\unlhd \cC$, their intersection $I\cap J$ is again a localizing ideal of $\cC$. Similarly, we define the sum $I+J$ to be the localizing ideal generated from $I$ and $J$ under colimits and tensoring with objects of $\cC$. The product $I\cdot J$ is the localizing ideal generated by the objects $X\otimes Y$ for $X\in I$ and $Y\in J$. 








\begin{prop}\label{principal_ideal_splits}
Let $\cC\in \calg(\Prst)$, let $M\in\cC$ and let $\nu\colon \Sigma^kM\to M$ be a self-map of some degree $k$. Then 
\[
\ideal{M/x}+\ideal{x^{-1}M} = \ideal{M}.
\]  
\end{prop}

\begin{proof}
On the one hand, $M/x$ and $x^{-1}M$ are colimits of copies of $M$ and therefore belong to $\ideal{M}$. On the other hand, the cofiber of the map $M\to x^{-1}M$ can be written as a colimit of the form 
\[
\Cofib(M\to x^{-1}M)\simeq\colim_{\ell \to \infty}(\Sigma^{-\ell k} M/x^{\ell}).
\]
Since each of the objects $\Sigma^{-\ell k} M/x^{\ell}$ belong to the thick subcategory generated by $M/x$, we deduce that $M\in \ideal{x^{-1} M} + \ideal{M/x}$ and the result follows. 
\end{proof}

\section{Ideals of Discrete Spectra}

We now consider ideals of spectra concentrated in degree $0$. The two cases of interest are $\ideal{\ZZ}$ and $\ideal{\FF_p}$. Note that by \Cref{principal_ideal_splits} we have 
\begin{equation}\label{eq:split_ideal_Z}
\ideal{\ZZ} = \ideal{\ZZ[\inv{p}]} + \ideal{\FF_p}.
\end{equation}

\begin{prop}\label{coacylcic_vs_bounded_above}
The ideal $\ideal{\ZZ}\subseteq \Sp$ contains all the bounded above spectra, while the ideal $\ideal{\FF_p}$ contains all bounded above spectra with homotopy groups which are $p$-power torsion.  
\end{prop}

\begin{proof}
Using the Whitehead tower, we can present a bounded above spectrum as a filtered colimit of bounded spectra. Then, every bounded spectrum is a finite extension of Eilenberg-MacLane spectra, which admit $\ZZ$-module structures and therefore belong to $\ideal{\ZZ}$. If all the homotopy groups are $p$-power torsion, then the Eilenberg-MacLane spectra that appear in this process are all built from extensions out of $\FF_p$-modules so we further get $X\in \ideal{\FF_p}$.   
\end{proof}

\Cref{coacylcic_vs_bounded_above} above shows that $L_{\ideal{\ZZ}}$ depends only on arbitrary connected covers. 
\begin{prop}\label{Z_coloc_connective_cover}
Let $X$ be a spectrum. The $m$-th connective cover map $\tau_{\ge m}X\to X$ induces an isomorphism 
\[
L_{\ideal{\ZZ}}(\tau_{\ge m} X)\iso L_{\ideal{\ZZ}}(X).
\]
If $\tau_{\le m-1}X$ has $p$-power torsion homotopy groups, then we also have
\[
L_{\ideal{\FF_p}}(\tau_{\ge m} X)\iso L_{\ideal{\FF_p}}(X).
\]
\end{prop}

\begin{proof}
The fiber of the map $\tau_{\ge m}X\to X$ is bounded above and hence belongs to the kernel of $L_{\ideal{\ZZ}}$ by the first part of \Cref{coacylcic_vs_bounded_above}. If $\tau_{\le m-1}X$ has only $p$-power torsion homotopy groups, then this fiber is in the kernel of $L_{\ideal{\FF_p}}$ as well by the second part of \Cref{coacylcic_vs_bounded_above}.
\end{proof}



\begin{prop}\label{connective_cover_ff}
The connective cover functor $\tau_{\ge 0}\colon L_{\ideal{\ZZ}}\Sp \to \Sp^\cn$ is a fully faithful right adjoint, with left adjoint given by the restriction of $L_{\ideal{\ZZ}}$ to connective spectra.
\end{prop}

\begin{proof}
The fact that it is a right adjoint with the claimed left adjoint follows by composing the adjunctions 
\[
L_{\ideal{\ZZ}}\Sp \simeq \Sp/\ideal{\ZZ} \adj \Sp \adj \Sp^\cn. 
\]
To show that $\tau_{\ge 0}$ is fully faithful on $L_{\ideal{\ZZ}}\Sp$, note that for $X,Y \in L_{\ideal{\ZZ}}\Sp$, by \Cref{Z_coloc_connective_cover}:
\[
\Map(X,Y)\simeq \Map(L_{\ideal{\ZZ}}\tau_{\ge 0}X,Y) \simeq \Map(\tau_{\ge 0}X,Y) \simeq \Map(\tau_{\ge 0}X,\tau_{\ge 0}Y) 
\]
\end{proof}

As a consequence of the proposition above, we can describe the essential image of the functor $\tau_{\ge 0}\colon L_{\ideal{\ZZ}}\colon \Sp \into \Sp^\cn$ as those connective spectra $X$ for which $X\iso \tau_{\ge 0}L_{\ideal{\ZZ}} X$. There is a more explicit description of these connective spectra:

\begin{prop} \label{connective_nonconnective_Z_coacyclic}
The following conditions for a connective spectrum $X$ are equivalent.
\begin{enumerate}
\item $X\iso \tau_{\ge 0}L_{\ideal{\ZZ}} X$. 
\item $\tau_{\ge -1} \hom(\ZZ,X) = 0$. 
\item There is a connective spectrum $Y$ such that $\Map(\ZZ,Y)=\pt$ and 
$
\tau_{\ge 0}\Sigma^{-1} Y\simeq X.
$ 
\end{enumerate}

\end{prop}

\begin{proof}
\underline{$(2)\iff(3)$}: If $\tau_{\ge -1}\hom(\ZZ,X)= 0$ we can choose $Y=\Sigma X$. Conversely, if there is such $Y$ then we have a fiber sequence 
\[
\Sigma^{-2}\pi_0 Y \to X \to \Sigma^{-1} Y 
\]
Applying $\hom(\ZZ,-)$ we obtain a fiber sequence 
\[
\Sigma^{-2}\hom(\ZZ,\pi_0 Y) \to \hom(\ZZ,X) \to \Sigma^{-1}\hom(\ZZ,Y).   
\]
Now, $\tau_{\ge -1}\Sigma^{-1}\hom(\ZZ,Y) \simeq 0$ by our assumption on $Y$, and $\tau_{\ge -1} \Sigma^{-2}\hom(\ZZ,\pi_0Y)=0$ since $\pi_0Y$ is discrete. Together, these imply that $\tau_{\ge -1}\hom(\ZZ,X) =0$ as well.

\underline{$(1)\implies (3)$}: If $\tau_{\ge 0}L_{\ideal{\ZZ}}X$ then we can choose $Y= \tau_{\ge 0} \Sigma L_{\ideal{\ZZ}}$. Indeed, this spectrum has has a contractible space of maps from $\ZZ$ since its the connective cover of an object of $L_{\ideal{\ZZ}}\Sp$.

\underline{$(2)\implies (1)$}:
Define a sequence of spectra iteratively as follows: Set $X_0 = X$ and $X_{k+1}$ to be the cofiber 
\[
\hom(\ZZ,X_k) \oto{\ev} X_k \too X_{k+1}.
\]
We shall show inductively that $\tau_{\ge -k-1}\hom(\ZZ,X_k)=0$. Indeed, let $C$ be the cofiber of the left unit map $\ZZ \to \ZZ\otimes \ZZ$. Since $\pi_1\ZZ\otimes \ZZ = 0$ and $ \pi_0 \ZZ \to \pi_0\ZZ\otimes \ZZ$ is injective, the $\ZZ$-module spectrum $C$ is $2$-connective. Applying $\hom(\ZZ,-)$ to the cofiber sequence defining $X_{k+1}$, and using  canonical identification $\hom(\ZZ,\hom(\ZZ,-))\simeq \hom(\ZZ\otimes \ZZ,-)$ we get a fiber sequence 
\[
\hom(\ZZ\otimes \ZZ,X_k)
\too \hom(\ZZ,X_{k})
 \too \hom(\ZZ,X_{k+1}), 
\]
from which we deduce that
\[
\hom(\ZZ,X_{k+1})\simeq \Sigma\hom(C,X_k)\simeq \Sigma\hom_\ZZ(C,\hom(\ZZ,X_k))\simeq \hom_\ZZ(C,\Sigma\hom(\ZZ,X_k)) 
\]
Since, by the inductive hypothesis, $\tau_{\ge -k}\Sigma\hom(\ZZ,X_k) = 0$ and $\tau_{\le 1}C =0$, we deduce that  
\[
\tau_{\ge -k-2}\hom(\ZZ,X_{k+1}) \simeq \tau_{\ge -k-2}\hom_\ZZ(C,\Sigma\hom(\ZZ,X_k)) = 0,
\]
hence verifying the inductive step. 

We now use this estimate to show that $\tau_{\ge 0}X_\infty = X$ and that $X_\infty \in L_{\ideal{\ZZ}} \Sp$. Note that these suffices to show that $X_\infty = L_{\ideal{\ZZ}}X$ by \Cref{connective_nonconnective_Z_coacyclic}. 
For the first claim, note that we have $\tau_{\ge 0}X_k\iso\tau_{\ge 0}X_{k+1}$ for all $k\in \NN$. Indeed, this follows from the fact that the fiber of the map $X_k\to X_{k+1}$ has vanishing homotopy groups in degrees $-1$ and above. Finally, we show that $X_\infty \in L_{\ideal{\ZZ}}\Sp$, or equivalently that $\hom(\ZZ,X_\infty)=0$.  By the fact that $\tau_{\ge -k-1}\hom(\ZZ,X_k)=0$, it would suffice to show that the assembly map 
\[
\colim\hom(\ZZ,X_k)\to \hom(\ZZ,\colim X_k)
\]
is an isomorphism. This assembly map participates in a cofiber sequence of assembly maps 
\[
\xymatrix{
 \colim\hom(\ZZ,\tau_{\ge 0}X_k)\ar[r]  \ar[d]&  \colim\hom(\ZZ,X_k)\ar[r] \ar[d] & \colim\hom(\ZZ,\tau_{\le -1}X_k)\ar[d] \\ 
 \hom(\ZZ,\colim\tau_{\ge 0}X_k)\ar[r] &  \hom(\ZZ,\colim X_k)\ar[r] & \hom(\ZZ,\colim\tau_{\le -1}X_k).
}
\]
The left vertical map is an isomorphism since the sequence $\tau_{\ge 0}X_0 \to \tau_{\ge 0}X_1 \to \dots$ is constant, by our proof of the first claim. The right vertical map is an isomorphism since $\ZZ$ is a spectrum of finite type, and hence $\hom(\ZZ,-)$ preserves filtered colimits of coconnective spectra. This shows that the middle vertical map is an isomorphism, concluding the proof. 
\end{proof}

\section{Chromatic Ideals}
 For $n\ge 1$ let\footnote{The notation $F(n)$ for a finite spectrum of type $n$ is non-standard. Note, for example, that if a Smith-Toda complex $V(n)$ exists then we can choose  $F(n+1):=V(n)$.} 
$F(n)$ be a finite spectrum of type $n$ at an (implicit) prime $p$. We also set $F(0):= \Sph$. Then, the quotient 
\[
\Sp/\ideal{F(n)} \simeq L_{n-1}^f\Sp
\] 
is the $(n-1)$-st finite localization, and the corresponding localization functor is denoted $L_{n-1}^f$. By the thick subcategory theorem of Hopkins and Smith \cite{nilp2}, after localizing at the prime $p$ the chains of ideals 
\[
\Sp = \ideal{F(0)} \supseteq \ideal{F(1)} \supseteq \ideal{F(2)}\supseteq \dots \supseteq \{0\}
\]
contain all the compactly generated localizing ideals in $\Sp_{(p)}$. 

Let $v_n\colon \Sigma^kF(n) \to F(n)$ be a $v_n$-self map, and let $T(n):=v_n^{-1}F(n)$ be the associated telescope. We define also $T(\infty):= \FF_p$. 
Using \Cref{principal_ideal_splits} iteratively we get, for $n\ge m$:
\begin{equation}\label{eq:split_chrom_ideals}
\ideal{F(m)} = \ideal{T(m),\dots,T(n),F(n+1)}. 
\end{equation}
Taking $m=0$ and applying $L_n^f$-localization, which precisely annihilates $F(n+1)$, we get that
\begin{equation}\label{eq:split_Lmf}
\ideal{T(0),\dots,T(n)} = L_n^f\Sp. 
\end{equation}

Motivated by this example, we consider the following localizing ideal.

\begin{defn}
We define the localizing ideal
\[
\mdef{\Spch} := \ideal{T(0),T(1),\dots,T(\infty)}\unlhd\Sp.
\]
\end{defn}
Thus, informally, $\Spch$ consists of those spectra which can be built from the ``chromatic cells'' $T(i)$ for $i=0,\dots.\infty$. 
\begin{rem}\label{rem:chromatic_modules}
Note that $\Spch$ is closed under all colimits in $\Sp$. In particular, if $R$ is an $\EE_1$-ring spectrum in $\Spch$ then so does every $R$-module spectrum. 
\end{rem}

We have the following source of examples of objects in $\Spch$.
\begin{defn} [{{cf. \cite[Definition 7.13]{ChromaticFourier}}}]
A spectrum $X$ is called \tdef{almost $L_n^f$-local} if $X\otimes F(n+1)$ is bounded above.  
\end{defn}

\begin{prop}\label{almost_local_chromatic}
If $X$ is almost $L_n^f$-local then $X\in \Spch$.
\end{prop}

\begin{proof}
Tensoring \Cref{eq:split_chrom_ideals} for $m=0$ with $X$ we get 
\[
\ideal{X} \subseteq \sum_{k=0}^n\ideal{X\otimes T(k)} + \ideal{X\otimes F(n+1)}\subseteq \sum_{k=0}^n\ideal{T(k)} + \ideal{X\otimes F(n+1)} 
\]
Since $\ideal{T(k)}\subseteq \Spch$, it remains to show that the same holds for $\ideal{X\otimes F(n+1)}$. But $X\otimes F(n+1)$ is bounded above by assumption, and all its homotopy groups are $p$-power torsion since $p$ is nilpotent on $F(n+1)$. By \Cref{coacylcic_vs_bounded_above}, these imply that 
\[
X\otimes F(n+1) \in \ideal{\FF_p}\subseteq \Spch.
\]
\end{proof}

By construction, the inclusion $L_n^f\Sp \into \Sp$ lands in $\Spch$, and we can restrict $L_n^f$ to a localization $L_n^f\colon \Spch \to L_n^f\Sp$. 
Moreover, 
the kernel of the localization $L_n^f$ simplifies after this restriction of its source.

\begin{prop}\label{kernel_Lnf_L_infty}
The kernel of $L_n^f\colon \Spch \to L_n^f\Sp$ is the localizing ideal $\ideal{T(n+1),\dots,T(\infty)}$ of $\Sp$. Namely, 
\[
\Spch \cap \ideal{F(n+1)} = \ideal{T(n+1),\dots,T(\infty)}.
\]
\end{prop}

\begin{proof}

Clearly, the right-hand side is contained in the left-hand side. On the other hand, 
note that $\Spch + \ideal{F(n+1)} = \Sp$ and hence:
\begin{align*}
\ideal{F(n+1)} \cap \Spch &= (\ideal{F(n+1)} \cap \Spch)\cdot(\Spch + \ideal{F(n+1)}) \\ 
&\subseteq \ideal{F(n+1)} \cdot \Spch = \sum_{m=0}^\infty \ideal{F(n+1)}\cdot \ideal{T(m)} = \sum_{m=0}^\infty \ideal{F(n+1)\otimes T(m)}.    
\end{align*}
Finally, if $m<n+1$ then $F(n+1)\otimes T(m) =0$ while if $m\ge n+1$ then $T(m)$ and $T(m)\otimes F(n+1)$ are both telescopes on $v_m$-self maps of spectra of type $m$ and hence generate the same thick subcategory, by the thick subcategory theorem. Hence, 
\[
\ideal{F(n+1)\otimes T(m)} = \begin{cases} 
\ideal{T(m)} & m\ge n+1 \\ 
0 & m\le n.
\end{cases}
\]
and thus 
\[
\ideal{F(n+1)} \cap \Spch \subseteq \sum_{m=0}^\infty \ideal{F(n+1)\otimes T(m)} = \sum_{m=n+1}^\infty \ideal{T(m)}.
\]
\end{proof}

\section{Ideals of Eilenberg-MacLane Spaces}
For a pointed space $A$, let 
$
\rsus{A}:= \Sigma^\infty A
$ 
be the reduced suspension spectrum of $A$. We denote $\mdef{\ideal{A}}:=\ideal{\rsus{A}}$. 
Let $B^nC_p$ be the Eilenberg-MacLane space with 
\[
\pi_iB^nC_p = \begin{cases}
C_p & i=n \\ 
0 & \text{else}.
\end{cases}
\]
Note that for $n\ge 1$ the spectrum $\rsus{B^nC_p}$ is $p$-local. 

Our goal from now on is to compare the ideals $\ideal{B^nC_p}$ with the compactly generated ideals $\ideal{F(n)}$. First, we show that like the latter, the former family form a descending chain.
\begin{prop}\label{EM_ideals_chain}
The ideals $\ideal{B^nC_p}$ satisfy 
\begin{equation}\label{eq:chain_EM_ideals}
\Sp = \ideal{B^0 C_p} \supseteq  \ideal{B^1C_p} \supseteq \ideal{B^2C_p}\dots \supseteq \ideal{B^nC_p}\supseteq \dots. 
\end{equation}

\end{prop}
\begin{proof}
First, recall that the reduced suspension spectrum functor $\rsus{-}\colon \Spc_\ast\to \Sp$ is colimit preserving. Now, the bar construction gives a colimit presentation 
\[
B^nC_p \simeq \colim_{[k]\in \Delta}(B^{n-1}C_p)^k \qin \Spc_\ast
\]
and hence 
\[
\rsus{B^nC_p} \simeq \colim_{[k]\in \Delta}\rsus{(B^{n-1}C_p)^k}.
\]
It will therefore suffices to show that $\rsus{(B^{n-1}C_p)^k}\in \ideal{B^{n-1}C_p}$ for all $k\in \NN$. For $k=1$ this is obvious, and for $k>1$ we can argue inductively using the cofiber sequences of relative suspension spectra 
\[
\rsus{B^{n-1}C_p}\too \rsus{(B^{n-1}C_p)^k}\too \sus{(B^{n-1}C_p)^{k-1}}\otimes \rsus{B^{n-1}C_p} 
\]
associated with the triple of spaces  $(\pt, (B^{n-1}C_p)^{k-1},(B^{n-1}C_p)^k)$.

\end{proof}
Next, we show that the chain of Eilenberg-MacLane ideals is majorized by the chain of ideals of finite spectra. 
\begin{prop}\label{EM_ideal_in_finite}
For all $n\in \NN$ there is an inclusion of ideals $\ideal{B^nC_p} \subseteq \ideal{F(n)}$. 
\end{prop}

\begin{proof}
Since $\ideal{F(n)}$ is the kernel of the localization $L_{n-1}^f$, the statement of the proposition is equivalent to the vanishing of
$
L_{n-1}^f\rsus{B^nC_p}.
$
This, in turn, follows from \cite[Corollary 5.3.7]{TeleAmbi}. 
\end{proof}
However, the gap between $\ideal{B^nC_p}$ and $\ideal{F(n)}$ is unseen by chromatic homotopy theory. Concretely, the following result holds. 
\begin{prop}\label{Tm_in_BmCp}
For all $n\in \NN$, 
\[
\ideal{B^nC_p}\cap \Spch = \ideal{F(n)}\cap \Spch.
\]
\end{prop}

\begin{proof}
First, by \Cref{EM_ideal_in_finite}, $\ideal{B^nC_p}\subseteq \ideal{F(n)}$. Thus,  it remains to show that 
\[
\ideal{F(n)}\cap \Spch\subseteq \ideal{B^nC_p}\cap \Spch. 
\]
Second, by \Cref{kernel_Lnf_L_infty},
\[
\ideal{F(n)}\cap \Spch = \ideal{T(n),\dots,T(\infty)}
\]
so the above inclusion is equivalent to the fact that $T(m)\in \ideal{B^nC_p}$ for $m\ge n$, including $m=\infty$. We begin with the case of finite $m$. Since by \Cref{EM_ideals_chain} there is a containment $\ideal{B^nC_p} \supseteq \ideal{B^mC_p}$, it would suffice to check this for $n=m$, that is, we only have to show that $T(n)\in \ideal{B^nC_p}$ for all $n\in \NN$.
The tensor product $T(n)\otimes \rsus{B^nC_p}$ belongs to $\ideal{B^nC_p}$, so it would suffice to exhibit $T(n)$ as a retract of this tensor product. By definition, $L_{T(n)} \rsus{B^nC_p}= \cyc[\Sph_{T(n)}]{p}{n}$ is the $p$-th cyclotomic extension of the $T(n)$-local sphere from \cite{carmeli2021chromatic}. By \cite[Proposition 5.2]{carmeli2021chromatic}, this cyclotomic extensions is an $\FF_p^\times$-Galois extension of $\Sph_{T(n)}$ and in particular
\[
\cyc[\Sph_{T(n)}]{p}{n}^{h\FF_p^\times}\simeq \Sph_{T(n)}.
\] 
Tensoring this isomorphism with $T(n)$ we get 
\[
(T(n)\otimes \rsus{B^nC_p})^{\FF_p^\times} \simeq T(n) \otimes  \rsus{B^nC_p})^{\FF_p^\times} \simeq T(n) \otimes  L_{T(n)}\rsus{B^nC_p})^{\FF_p^\times} \simeq T(n),
\]
where we used the facts that tensoring with $T(n)$ preserves $\FF_p^\times$-fixed points and that it is invariant under $T(n)$-localization.
Finally, since $T(n)\otimes \rsus{B^nC_p}$ is $p$-complete and $\FF_p^\times$ is of order prime to $p$, the fixed points for the action are a retract of the underlying spectrum, that is, $T(n)$ is a retract of $T(n)\otimes \rsus{B^nC_p}$. 

It remains to show that $T(\infty):= \FF_p \in \ideal{B^nC_p}$. Similarly to the finite case, it would suffice to show that $\FF_p$ is, up to suspension, a retract of $\rsus{B^nC_p}\otimes \FF_p$. But $\rsus{B^nC_p}\otimes \FF_p$ splits into a non-empty sum of suspensions of copies of $\FF_p$, because it is a non-zero $\FF_p$-module spectrum.
\end{proof} 


\begin{rem}
Robert Burklund explained to the author that a variant of an unpublished argument of Hopkins and Smith for the equality of the Bausfield classes of $\rsus{BC_p}$ and $F(1)$ gives the equality of ideals $\ideal{BC_p}= \ideal{F(1)}$. Consequently, for $n=1$ the localizing subcategory $L_{\ideal{BC_p}}\Sp \subseteq \Sp$ is the $\infty$-category of $\Sph[\inv{p}]$-module spectra. Such a simple description is impossible for higher values of $m$, though, as the following example shows.
\end{rem}

\begin{example}\label{ex:Lee}
The ideals $\ideal{B^2C_p}$ and $\ideal{F(2)}$ are different. In fact, 
the $p$-complete sphere spectrum $\Sph_p$ belongs to $L_{\ideal{B^2C_p}}\Sp$. Indeed, this statement is equivalent to the claim that pullback morphism $\pi^*\colon \Sph_p \to \Sph_p^{B^2C_p}$ along the terminal map $\pi\colon B^2C_p\to \pt$ is an isomorphism, which is a special case of \cite[Theorem 2.5]{lee1992stable}. 
\end{example}

It will be convenient to reformulate \Cref{Tm_in_BmCp} in terms of localization functors and mapping spectra.

\begin{cor} \label{chromatic_EM_loc}
For all $n\ge 1$, the localization $L_{n-1}^f$ is a further localization of $L_{\ideal{B^nC_p}}$. On the other hand, they agree when restricted to $\Spch$. Namely, for $X\in \Spch$ we have 
\[
L_{\ideal{B^nC_p}}X \simeq L_{n-1}^fX.
\] 
\end{cor}
\begin{proof}
Recall that a localization $L$ is a further localization of $L'$ if and only if the kernel of $L$ contains that of $L'$. Thus, the first part follows from \Cref{EM_ideal_in_finite} applied to the localizations $L_{\ideal{B^nC_p}}$ and $L_{n-1}^f$ of $\Sp$, while the second follows from \Cref{Tm_in_BmCp} applied to the corresponding localizations of $\Spch$.  
\end{proof}

\begin{cor}\label{almost_Lnf_EM_loc}
Let $X$ be an almost $L_n^f$-local spectrum for some $n\in \NN$. Then 
\[
L_{\ideal{B^nC_p}}X\simeq L_{n-1}^fX.
\] 
\end{cor}

\begin{proof}
By \Cref{almost_local_chromatic}, $X\in \Spch$, so the result follows from \Cref{chromatic_EM_loc}. 
\end{proof}

\begin{cor}\label{maps_from_chromatic_to_EM_colocal}
Let $Y\in L_{\ideal{B^nC_p}}\Sp$ and let $X\in \Spch$ (e.g. an almost $L_m^f$-local spectrum for some $m\in \NN$). Then, 
\[
\hom(X,Y)\simeq \hom(L_{n-1}^fX,Y).
\]
\end{cor}

\begin{proof}
Since $Y\in L_{\ideal{B^nC_p}}\Sp$, we have
\[
\hom(X,Y)\simeq \hom(L_{\ideal{B^nC_p}}X,Y).
\]
But $L_{\ideal{B^nC_p}}X\simeq L_{n-1}^fX$ by \Cref{chromatic_EM_loc} and the result follows.
\end{proof}

\section{Maps into the Units of the Sphere Spectrum}
We shall now consider the implications of the theory developed in the previous section to the spectrum $\Sph_p^\times$. 
Note that 
\begin{equation} \label{eq:units_split}
\Sph_p^\times\simeq (\Sph_p^\times)^\wedge_p \oplus \FF_p^\times, 
\end{equation}
so studyings maps of connective spectra into $\Sph_p^\times$ and its $p$-completion are essentially the same.  
We first show that it is safe to replace the $p$-completion with a suitable delooping, so that we can work stably. 

\begin{prop}
The canonical map 
$
(\Sph_p^\times)^\wedge_p \to \tau_{\ge 0}L_{\ideal{\ZZ}}(\Sph_p^\times)^\wedge_p
$
is an isomorphism.
\end{prop}

\begin{proof}
By \Cref{connective_nonconnective_Z_coacyclic}, it would suffice to show that there is a connective spectrum $Y$ such that $\tau_{\ge 0}\Sigma^{-1}Y\simeq (\Sph_p^\times)^\wedge_p$ and $\Map(\ZZ,Y)=\pt$. In fact, we can choose $Y$ to be the $p$-completion of the Picard spectrum $\pic(\Sph_p)$. Indeed, it follows from \cite[Corollary 1.2]{carmeli2023strict} that  $\Map(\ZZ,\pic(\Sph_p)^\wedge_p)\simeq \pt$, and since $\pic(\Sph_p)$ has finitely generated homotopy groups, taking connected cover commutes with $p$-completion for it, so that $\tau_{\ge 0}\Sigma^{-1}Y \simeq (\Sph_p^\times)^\wedge_p.$
\end{proof}

From now on, we shall work with this non-connective spectrum, which we abbreviate as follows:

\begin{defn}
We define the spectrum 
$
\mdef{\br}:= L_{\ideal{\ZZ}}(\Sph_p^\times)^\wedge_p.
$
\end{defn}
Hence, 
\[
\tau_{\ge 0}\br \simeq (\Sph_p^\times)^\wedge_p,
\]
and by definition, $\br \in L_{\ideal{\ZZ}}\Sp$. 
We shall need two other such locality properties of $\br$.
\begin{prop}\label{br_p_complete}
The spectrum $\br$ is $p$-complete.  
\end{prop}

\begin{proof}
We have to show that $\Map(X,\br) \simeq \pt$ for every spectrum $X$ on which $p$ is invertible. Since $\br\in L_{\ideal{\ZZ}} \Sp$, by \Cref{connective_cover_ff} we have the following isomorphisms: 
\[
\Map(X,\br) \simeq \Map(\tau_{\ge 0}X,\tau_{\ge 0}\br)\simeq \Map(\tau_{\ge 0}X,(\Sph_p^\times)^\wedge_p). 
\]
Since $p$ is invertible also on $\tau_{\ge 0}X$ and $(\Sph_p^\times)^\wedge_p$ is $p$-complete, the latter mapping space is contractible and the result follows. 
\end{proof} 
 
By Lee's theorem, no non-zero maps of connective spectra $\rsus{B^2C_p} \to \Sph_p^\times$ exist. This formally implies an analogous property for our spectrum $\br$.
\begin{prop}
The spectrum $\br$ satisfies
$
\br\in L_{\ideal{B^2C_p}}\Sp.
$
\end{prop}

\begin{proof}
We have to show that $\hom(\rsus{B^2C_p},\br) = 0$. Now,
\begin{align*}
\Omega^\infty\hom(\rsus{B^2C_p},\br)&\simeq \Map(\rsus{B^2C_p},\br)\simeq \Map(\rsus{B^2C_p},\tau_{\ge 0} \br) \\ 
&\simeq \Map(\rsus{B^2C_p},(\Sph_p^\times)^\wedge_p) \simeq \Map(\rsus{B^2C_p},\Sph_p^\times) 
\\&\simeq 
\Fib((\Sph_p^\times)^{B^2C_p} \to \Sph_p^\times) \simeq \Fib((\Sph_p^{B^2C_p})^\times \to \Sph_p^\times), 
\end{align*}
where the second isomorphism is since $\rsus{B^2C_p}$ is connective, the third by \Cref{coacylcic_vs_bounded_above}, and the fourth since $\rsus{B^2C_p}\in \ideal{F(1)}$, so maps from it to a spectrum and to its $p$-completion are the same.

By Lee's theorem (see \Cref{ex:Lee}), $\Sph_p\iso \Sph_p^{B^2C_p}$ so the fiber of the induced map  on units 
$(\Sph_p^{B^2C_p})^\times \to \Sph_p^\times
$ 
is contractible. We deduce that $\Omega^\infty \hom(\rsus{B^2C_p},\br) = \pt$, so that in particular $\hom(\rsus{B^2C_p},\br)$ is bounded above. However, by definition, $\br\in L_{\ideal{\ZZ}}\Sp$. Since $\ideal{\ZZ}$ is closed under tensoring with $\rsus{B^2C_p}$, the localizing subcategory $L_{\ideal{\ZZ}}\Sp$ is closed under applying $\hom(\rsus{B^2C_p},-)$ and we deduce that 
$
\hom(\rsus{B^2C_p},\br)\in L_{\ideal{\ZZ}}\Sp.
$
Together, by \Cref{coacylcic_vs_bounded_above} these imply that $\hom(\rsus{B^2C_p},\br)=0$.   
\end{proof}

From the relationship between the ideal of $B^2C_p$ and the chromatic localizations of the previous section, we deduce the following:

\begin{cor}\label{maps_into_G_L_1}
With the notation as above, 
for $X\in \Spch$ we have 
\[
\Map(X,\br)\simeq \Map(L_{K(1)}X,\br).  
\]
\end{cor}

\begin{proof}
Since $\br\in L_{\ideal{B^2C_p}}\Sp$, we deduce from \Cref{maps_from_chromatic_to_EM_colocal} that 
\[
\Map(X,\br)\simeq \Map(L_1^f X,\br).
\]
Since by \Cref{br_p_complete}, $\br$ is also $p$-complete and $L_{K(1)}X$ is the $p$-completion of $L_1^fX$, we obtain
\[
\Map(L_1^f X,\br) \simeq \Map((L_1^f X)^\wedge_p,\br) \simeq \Map(L_{K(1)}X,\br).
\]
\end{proof}

For connective $X$, this immediately translates to our main result regarding  $\Sph_p^\times$. 

\begin{thm} \label{main_thm}
Let $X$ be a connective spectrum in $\Spch$. Then,
\[
\Map(X,\Sph_p^\times) \simeq \Hom_{\Ab}(\pi_0X,\FF_p^\times) \times  \Map(\tau_{\ge 0}L_{K(1)} X,(\Sph_p^\times)^\wedge_p). 
\]
\end{thm}

\begin{proof}
Using the decomposition 
\[
\Sph_p^\times \simeq (\Sph_p^\times)^\wedge_p \oplus \FF_p^\times
\]
we immediately reduce the statement to
\[
\Map(X,(\Sph_p^\times)^\wedge_p)\simeq \Map(\tau_{\ge 0}L_{K(1)}X,(\Sph_p^\times)^\wedge_p)
\] 
This, in turn, follows from the chain of isomorphisms 
\begin{align*}
\Map(X,(\Sph_p^\times)^\wedge_p) &\simeq \Map(X,\br) \\ 
&\simeq \Map(L_{K(1)}X,\br)
\\ &\simeq \Map(\tau_{\ge 0}L_{K(1)}X,\tau_{\ge 0} \br) \\
&\simeq \Map(\tau_{\ge 0} L_{K(1)}X,(\Sph_p^\times)^\wedge_p)
\end{align*}
where the first isomorphism is since $X$ is connective and $\tau_{\ge 0}\br\simeq (\Sph_p^\times)^\wedge_p$, the second is \Cref{maps_into_G_L_1}, the third follows from \Cref{connective_cover_ff} and the fourth is again by
$\tau_{\ge 0}\br\simeq (\Sph_p^\times)^\wedge_p$. 
\end{proof}

\begin{rem}
Expanding \Cref{rem:pic_instead_of_gl},
it is natural to ask if one can obtain a similar result for $\pic(\Sph_p)$. Note that our result for $\Sph_p^\times$ depends on a vanishing result for $\pic(\Sph_p)^\wedge_p$, namely, that it receives no maps from $\ZZ$. Similarly, we could lift the result from $\Sph_p^\times$ to $\pic(\Sph_p)$ if we knew that there are no maps from $\ZZ$ to the $p$-completion of the Brauer spectrum of $\Sph_p$. This is a result of a work in progress with Burklund and Ramzi. 
\end{rem} 

\section{Applications}

To demonstrate the utility of \Cref{main_thm}, we list some examples of familiar spectra belonging to $\Spch$, for which \Cref{itro:main} applies.
First, observe that the connective cover of every $L_n^f$-local spectrum is in $\Spch$. For example, this applies to Lubin-Tate theories.

\begin{prop}
Let $E_n$ be a Lubin--Tate theory associated with a formal group law of height $n$ over a perfect ring $\kappa$ of characteristic $p$, so that 
$\pi_0E_n\simeq \WW(\kappa)[[u_1,\dots,u_n]]^\wedge_p$. 
Every map $\tau_{\ge 0}E_n\to \Sph_p^\times$ factors uniquely through the map 
$
\tau_{\ge 0}E_n \to \tau_{\ge 0}E_n[u_1^{-1}]^\wedge_p:
$
\[
\xymatrix{
\tau_{\ge 0}E_n\ar[rr]\ar[d] && \Sph_p^\times \\ 
\tau_{\ge 0}E_n[u_1^{-1}]^\wedge_p\ar@{-->}[rru] &&
}
\]
\end{prop}

\begin{proof}
This follows from \Cref{main_thm} by observing that $\pi_0E_n/(p-1) = 0$ and hence it has no non-zero maps to $\FF_p^\times$, and that $L_{K(1)}E_n\simeq E_n[u_1^{-1}]^\wedge_p$. 
\end{proof}

\begin{rem}
Note that $\tau_{\ge 0}E_n[u_1^{-1}]^\wedge_p$ is a module over connective, $p$-completed complex $K$-theory $\ku^\wedge_p$. Hence, every map 
$\tau_{\ge 0}E_n \to \Sph_p^\times$ have to factor through the evaluation map 
\[
\hom(\ku^\wedge_p,\Sph_p^\times) \to \Sph_p^\times,
\]
i.e., through the (yet unknown) spectrum of ''complex $J$-homomorphisms''. 
\end{rem}

Next, we discuss maps from $K$-theory spectra  (and their suspensions) into $\Sph_p^\times$. 
The following result is an immediate consequence of Waldhausen's interpretation of the Quillen-Lichtenbaum conjecture.
\begin{prop}\label{K-theory-disc-chrom}
Let $R$ be a $\ZZ$-algebra in $\Sp$ (for example an ordinary associative ring). Then $K(R)\in \Spch$. Consequently, for all $t\in \NN$,
\[
\Map(\tau_{\ge 0}\Sigma^{-t}K(R),\Sph_p^\times)
\simeq \Hom(\pi_tK(R),\FF_p^\times) \times \Map(\tau_{\ge 0}\Sigma^{-t}L_{K(1)}K(R),\Sph_p^\times). 
\]
\end{prop}

\begin{proof}
First, observe that since bounded above spectra are in $\Spch$, if $K(R)\in \Spch$ so does $\tau_{\ge 0}\Sigma^{-t} K(R)$. Hence, the ``consequently'' part indeed follows from the first claim. Next, since $K(R)$ is a module over $K(\ZZ)$, by \Cref{rem:chromatic_modules} it suffices to show that $K(\ZZ)\in \Spch$. As shown in \cite[Appendix 4]{WaldKchrom}, the validity of the Quillen-Lichtenbaum conjecture for the ring $\ZZ$ (\cite{RogQLTwo,KahnQLTwo,voevodsky2011motivic}) implies that the map $K(\ZZ)\to L_{K(1)}K(\ZZ)$ has bounded above fiber, which in turn implies that  $K(\ZZ)\in \Spch$ by \Cref{coacylcic_vs_bounded_above} (note that $\ideal{\ZZ}\subseteq \Spch$).
\end{proof}





Finally, in \cite{hahn2020redshift}, Hahn and Wilson constructed $\EE_3$-$\MU$-algebra structures on the truncated Brown-Peterson spectra $\BP{n}$. For these $\EE_3$-algebras, they established a higher height analog of the Quillen-Lichtenbaum conjecture, and hence we have an analog of \Cref{K-theory-disc-chrom}.

\begin{prop}\label{HW-chromatic}
Let $\BP{n}$ be the $\EE_3$-algebra as in \cite{hahn2020redshift}.
For every $\BP{n}$-algebra $R$, the spectrum $K(R)$ belongs to $\Spch$. Consequently, for all $t\in \NN$, 
\[
\Map(\tau_{\ge 0}\Sigma^{-t}K(R),\Sph_p^\times) \simeq \Hom(\pi_tK(R),\FF_p^\times) \times  \Map(\tau_{\ge 0}\Sigma^{-t}L_{K(1)}K(R),\Sph_p^\times). 
\]
\end{prop}

\begin{proof}
As in the proof of \Cref{K-theory-disc-chrom}, we can reduce to show that $K(\BP{n})$ belongs to $\Spch$. By \cite[Corollary 3.4.4]{hahn2020redshift} the fiber of the map $K(R)_{(p)}\to L_{n+1}^fK(R)_{(p)}$ is bounded above. Since the target is in $\Spch$, so does the source. Thus, $K(R)\otimes F(1)\simeq K(R)_{(p)}\otimes F(1) \in \Spch$, which implies that $K(R)$ belongs to $\Spch$ as well, since $\ideal{K(\ZZ)}\subseteq \ideal{K(\ZZ)[\inv{p}]} + \ideal{K(\ZZ)\otimes F(1)}$ (see \Cref{principal_ideal_splits}). 
\end{proof}

By the purity result of \cite[Theorem A]{land2020purity},
$L_{K(1)}K(R)\simeq L_{K(1)}K(L_1^fR)$, and similarly 
$\pi_tK(R)/(p-1)\simeq \pi_tK(L_1^fR)/(p-1)$. Combining with \Cref{HW-chromatic}, we deduce the following:
\begin{cor}
With the settings of \Cref{HW-chromatic},
the map $R\to L_1^fR$ induces isomorphisms 
\[
\Map(\tau_{\ge 0}\Sigma^{-t}K(R),\Sph_{p}^\times) \simeq \Map(\tau_{\ge 0}\Sigma^{-t}K(L_1^fR),\Sph_{p}^\times)
\]
for all $t\in \NN$. 
\end{cor}

\bibliographystyle{alpha}
\bibliography{strict}

\begin{thebibliography}{LMMT20}

\bibitem[Ada63]{AdamsJI}
John~F. Adams.
\newblock {On the groups $J(X)$-I}.
\newblock {\em Topology}, 2(3):181--195, 1963.

\bibitem[Ada65a]{AdamsJII}
John~F. Adams.
\newblock {On the groups $J(X)$-II}.
\newblock {\em Topology}, 3:137--171, 1965.

\bibitem[Ada65b]{AdamsJIII}
John~F. Adams.
\newblock {On the groups $J(X)$-III}.
\newblock {\em Topology}, 3(3):193--222, 1965.

\bibitem[AGM92]{adams1992segal}
John~F. Adams, Jeremy~H. Gunawardena, and Heynes Miller.
\newblock The segal conjecture for elementary abelian p-groups.
\newblock {\em The Selected Works of J. Frank Adams: Volume 2}, 2(4):143, 1992.

\bibitem[BCSY22]{ChromaticFourier}
Tobias Barthel, Shachar Carmeli, Tomer~M. Schlank, and Lior Yanovski.
\newblock The chromatic fourier transform.
\newblock {\em arXiv preprint arXiv:2210.12822}, 2022.

\bibitem[BR22]{bruner2022adams}
Robert Bruner and John Rognes.
\newblock {The Adams spectral sequence for the image-of-j spectrum}.
\newblock {\em Transactions of the American Mathematical Society}, 375(08):5803--5827, 2022.

\bibitem[BSY22]{Nulls}
Robert Burklund, Tomer~M. Schlank, and Allen Yuan.
\newblock The chromatic nullstellensatz.
\newblock {\em arXiv preprint arXiv:2207.09929}, 2022.

\bibitem[Car84]{carlsson1984equivariant}
Gunnar Carlsson.
\newblock {Equivariant stable homotopy and Segal's Burnside ring conjecture}.
\newblock {\em Annals of Mathematics}, pages 189--224, 1984.

\bibitem[Car23]{carmeli2023strict}
Shachar Carmeli.
\newblock On the strict picard spectrum of commutative ring spectra.
\newblock {\em Compositio Mathematica}, 159(9):1872--1897, 2023.

\bibitem[Cla11]{pAdicJ}
Dustin Clausen.
\newblock {$p$-adic $J$-homomorphisms and a product formula}.
\newblock {\em arXiv preprint arXiv:1110.5851}, 2011.

\bibitem[CMNN20]{clausen2020descent}
Dustin Clausen, Akhil Mathew, Niko Naumann, and Justin Noel.
\newblock {Descent and vanishing in chromatic algebraic $K$-theory via group actions}.
\newblock {\em arXiv preprint arXiv:2011.08233}, 2020.

\bibitem[CSY18]{TeleAmbi}
Shachar Carmeli, Tomer~M. Schlank, and Lior Yanovski.
\newblock Ambidexterity in chromatic homotopy theory.
\newblock {\em arXiv preprint arXiv:1811.02057}, 2018.

\bibitem[CSY21]{carmeli2021chromatic}
Shachar Carmeli, Tomer~M. Schlank, and Lior Yanovski.
\newblock Chromatic cyclotomic extensions.
\newblock {\em arXiv preprint arXiv:2103.02471}, 2021.

\bibitem[HS98]{nilp2}
Michael~J. Hopkins and Jeffrey~H. Smith.
\newblock Nilpotence and stable homotopy theory {II}.
\newblock {\em Ann. of Math. (2)}, 148(1):1--49, 1998.

\bibitem[HW20]{hahn2020redshift}
Jeremy Hahn and Dylan Wilson.
\newblock {Redshift and multiplication for truncated Brown-Peterson spectra}.
\newblock {\em arXiv preprint arXiv:2012.00864}, 2020.

\bibitem[Kah97]{KahnQLTwo}
Bruno Kahn.
\newblock {The Quillen-Lichtenbaum conjecture at the prime $2$}.
\newblock {\em preprint}, 1997.

\bibitem[Lee92]{lee1992stable}
Chun-Nip Lee.
\newblock {On stable maps from Eilenberg-MacLane spaces to classifying spaces}.
\newblock {\em American Journal of Mathematics}, 114(2):405--412, 1992.

\bibitem[Lin80]{lin1980conjectures}
Wen-Hsiung Lin.
\newblock On conjectures of mahowald, segal and sullivan.
\newblock In {\em Mathematical Proceedings of the Cambridge Philosophical Society}, volume~87, pages 449--458. Cambridge University Press, 1980.

\bibitem[LMMT20]{land2020purity}
Markus Land, Akhil Mathew, Lennart Meier, and Georg Tamme.
\newblock {Purity in chromatically localized algebraic $K$-theory}.
\newblock {\em arXiv preprint arXiv:2001.10425}, 2020.

\bibitem[Lur]{JacobChrom35}
Jacob Lurie.
\newblock \href{https://people.math.harvard.edu/~lurie/252x.html}{Chromatic Homotopy Theory}, lecture 35.

\bibitem[Mah70]{mahowald1970order}
Mark Mahowald.
\newblock The order of the image of the j-homomorphism.
\newblock 1970.

\bibitem[Rog14]{Rognes_Redshift}
John Rognes.
\newblock Chromatic redshift.
\newblock {\em arXiv preprint arXiv:1403.4838}, 2014.

\bibitem[RW97]{RogQLTwo}
John Rognes and Chuck Weibel.
\newblock {Two-primary algebraic $K$-theory of rings of integers in number fields}.
\newblock {\em preprint}, 1997.

\bibitem[Voe11]{voevodsky2011motivic}
Vladimir Voevodsky.
\newblock {On motivic cohomology with $\mathbb{Z}/\ell$-coefficients}.
\newblock {\em Annals of mathematics}, pages 401--438, 2011.

\bibitem[Wal82]{WaldKchrom}
Friedhelm Waldhausen.
\newblock {Algebraic $K$-theory of spaces, localization, and the chromatic filtration of stable homotopy}.
\newblock {\em Algebraic topology, Aarhus}, pages 173--195, 1982.

\bibitem[Whi42]{whiteheadJhomomorphism}
George~W. Whitehead.
\newblock On the homotopy groups of spheres and rotation groups.
\newblock {\em Annals of Mathematics}, pages 634--640, 1942.

\end{thebibliography}

\end{document}